\documentclass[letterpaper,12pt]{article}

\usepackage[utf8x]{inputenc}
\usepackage{latexsym,fullpage}
\usepackage{graphics,graphicx}
\usepackage{subfigure}
\usepackage{psfrag}
\usepackage{epic,eepic,color,verbatim}
\usepackage{euscript,amssymb,amsthm}
\usepackage{amsfonts}
\usepackage{amsmath}
\usepackage{amssymb}
\usepackage{amscd}
\usepackage{amsthm}
\usepackage{mathrsfs}
\usepackage{cite}
\usepackage{pstricks, pstricks-add, pst-math,pst-node}
\usepackage{pst-xkey}
\usepackage{calc}
\usepackage{hyperref}
\usepackage{enumerate}

\usepackage{verbatim}

\usepackage{soul}

\usepackage{authblk}

\DeclareSymbolFont{slantedbx}{OT1}{cmr}{bx}{sl}
\DeclareSymbolFontAlphabet\mathslbx{slantedbx}
\DeclareSymbolFont{AMSa}{U}{msa}{m}{n} \DeclareMathSymbol{\upharpoonright} {\mathrel}{AMSa}{"16}

\def\oG{{\overline{G}}}

\def\ol#1{{\overline{#1}}}
\def\gd{{G^*_{u,v}}}
\def\gde{{G_{u,v}}}
\def\ff{{\cal F}}
\def\gg{{\cal G}}
\def\ii{{\cal I}}
\def\dd{{\cal D}}
\def\w{{\ol{\lambda}}}
\def\wuvd{{\lambda^*}}
\def\wuv{{\lambda}}

\def\ignore#1{{\vglue 0.05 cm}}

\newtheorem{thm}{Theorem} 
\newtheorem{thm*}{Theorem*}
\newtheorem{lem}[thm]{Lemma} 

\newtheorem{conj}[thm]{Conjecture} 

\newtheorem{proposition}[thm]{Proposition}

\newtheorem{question}[thm]{Question}

\DeclareMathOperator{\cro}{\rm cr}

\def\real{{\mathbb{R}}}

\title{\bf Making a graph crossing-critical \\ by multiplying its edges}

\author[1]{Laurent Beaudou}
\author[2]{C\'esar Hern\'andez--V\'elez}
\author[,2]{Gelasio Salazar\thanks{Supported by CONACYT Grant 106432.}}
\affil[1]{LIMOS, Universit\'e Blaise Pascal, Clermont-Ferrand, France}
\affil[2]{Instituto de F\'\i sica, Universidad Aut\'onoma de San Luis Potos\'{\i}. San Luis Potos\'{\i}, Mexico 78000}

\begin{document}
\maketitle

\begin{abstract}
A graph is {\em crossing-critical} 
if the removal of any of its edges
decreases its crossing number. 
This work is motivated by the following question: to what extent
is crossing-criticality a property that is inherent to the structure
of a graph, and to what extent can it be induced on a noncritical
graph by
multiplying (all or some of) its edges?
It is shown that if a nonplanar graph $G$ is
obtained by adding an edge to a cubic polyhedral graph, and $G$ is
sufficiently connected, then $G$ can be made crossing-critical by a
suitable multiplication of edges.
\end{abstract}



\section{Introduction}

This work is motivated  by the recent breakthrough constructions by DeVos, Mohar and \v{S}\'amal~\cite{devosetal} and Dvo\v{r}\'ak and Mohar~\cite{dvorakmohar}, which settled 
two important
crossing numbers questions. The graphs constructed in~\cite{devosetal}
and~\cite{dvorakmohar} use weighted (or ``thick'') edges. A graph
with weighted edges
can be naturally transformed into an ordinary graph by substituting
weighted edges by  multiedges
(recall that a {\em multiedge} is a set of edges with
the same pair of endvertices). If one wishes to avoid multigraphs, one
can always substitute a weight $t$ edge by a $K_{2,t}$, but still the
resulting graph is homeomorphic to a multigraph. Sometimes (as
in~\cite{devosetal}) one can afford to substitute each weighted edge by a
slightly richer structure (such as a graph obtained from $K_{2,t}$ by
joining the degree $2$ vertices with a path), but sometimes (as
in~\cite{dvorakmohar}) one is concerned with criticality properties,
and so no such superfluous edges may be added.  In any case, the use of
weighted edges is crucial. 

After trying unsuccessfully to come up with graphs with similar
crossing number properties as those presented in~\cite{devosetal}
and~\cite{dvorakmohar}, while avoiding the use of weighted edges, we were left with a wide open question: in the realm of
crossing numbers, 
more specifically on crossing-criticality issues, 
to what extent does it make a difference to allow (equivalently, to
forbid) weighted edges (or, for that matter, multiedges)?








Recall that the {\em crossing number} $\cro(G)$ of a graph $G$ is the minimum number of pairwise intersections of edges in a drawing of $G$ in the plane. 
An edge $e$ of $G$ is {\em crossing-critical} if $\cro(G-e) < \cro(G)$. If all edges of $G$ are crossing-critical, then  $G$ itself is {\em crossing-critical}. A crossing-critical graph seems naturally more interesting
than a graph with some not crossing-critical edges, since a graph of the latter kind contains a proper subgraph that has all the relevant information from the crossing numbers point of view.  


Earlier constructions of infinite families of crossing-critical graphs made essential use
of multiple edges~\cite{siran0}. On the other hand, constructions such
the ones given by Kochol~\cite{kochol}, Hlin\v{e}n\'y~\cite{hli1}, and
Bokal~\cite{bokal}
deal exclusively with simple graphs. 

We ask to what extent crossing-criticality is an inherent structural property of a graph, and to what extent crossing-criticality can be induced by multiplying the edges of a (noncritical) graph. Let $G, H$ be graphs. We say that $G$ {\em is obtained by multiplying edges of   $H$} if $H$ is a subgraph of $G$ and, for every edge of $G$, there is an edge of $H$ with the same endvertices.

\begin{question}\label{question1}
When can a graph be made crossing-critical by multiplying edges? That is, given a (noncritical) graph $H$, when does there exist a crossing-critical graph $G$ that is obtained by multiplying edges of $H$?
\end{question}


Our universe of interest is, of course, the set of nonplanar graphs, since a planar graph obviously remains planar after multiplying any or all of its edges.

We show that a large, interesting family of nonplanar graphs satisfy the property in Question~\ref{question1}. A nonplanar graph $G$ is {\em near-planar} if it has an edge $e$ such that $G-e$ is planar. Near-planar graphs constitute a natural family of nonplanar graphs. Any thought to the effect that crossing number problems might become easy when restricted to near-planar graphs is put definitely to rest by the recent  proof by Cabello and Mohar that {\sc   CrossingNumber} is NP-Hard for near-planar graphs~\cite{cabellomohar1}.

Following Geelen et al.~\cite{Ge} (who define internally
$3$-connectedness for matroids), a graph $G$ is \emph{internally $3$-connected} if $G$ is
simple and $2$-connected, and for every separation $(G_1, G_2)$ of $G$ of
order two, either $|E(G_1)| \leq 2$ or $|E(G_2)| \leq 2$. Hence,
internally $3$-connected graphs are those that can be obtained from a
$3$-connected graph by subdividing its edges, with the condition that
no edge can be subdivided more than once.

Our main result is that any adequately connected, near-planar graph
$G$ obtained by adding an edge to a cubic polyhedral (i.e., planar and
$3$-connected) graph, belongs to the class alluded to in
Question~\ref{question1}.

\begin{thm}\label{maintheorem}
Let $G$ be a near-planar simple graph, with an edge $uv$ such that
$G-uv$ is a cubic polyhedral graph.  Suppose that $G -\{u,v\}$ is 
internally $3$-connected. Then there exists a crossing-critical
graph that is obtained by multiplying edges of $G$.
\end{thm}



We note that some connectivity assumption is needed in order to
guarantee that a nonplanar graph can be made crossing-critical by
multiplying edges. To see this, consider a graph $G$ which is the
$1$-sum of a nonplanar graph $G_1$ plus a planar graph $G_2$. Since
crossing number is additive on the blocks of a graph, it is easy to
see that $G$ cannot be made crossing-critical by multiplying edges.

An important ingredient in the proof of Theorem~\ref{maintheorem} is
the following, somewhat curious statement for which we could not find
any reference in the literature. We recall that a {\em weighted graph}
is a pair $(G,w)$, where $G$ is a graph and $w$ (the {\em weight
  assignment}) is a map that assigns to each edge $e$ of $G$ a number
$w(e)$, the {\em weight} of $e$. The {\em length} of a path in a
weighted graph is the sum of the weights of the edges in the path. If
$u,v$ are vertices of $G$, then the {\em distance} $d_{w}(u,v)$ {\em
  from $u$ to $v$} ({\em under $w$}) is the length of a minimum length
(also called a {\em shortest}) {$uv$-path}.  The weight assignment $w$
is {\em positive} if $w(e) >0$ for every edge $e$ of $G$, and it is
{\em integer} if each $w(e)$ is an integer.

\begin{lem}\label{thelemma}
Let $G$ be a $2$-connected loopless graph, and let $u,v$ be distinct vertices of
$G$. Then there is a positive integer 
weight assignment
such that every edge of $G$ belongs to a shortest $uv$-path.
\end{lem}



The rest of this paper is structured as follows.

We prove the auxiliary Lemma~\ref{thelemma} in Section~\ref{prooflemma}. We then proceed to reformulate Theorem~\ref{maintheorem} in terms of weighted graphs. This simply consists on replacing a multigraph with a weighted simple graph so that the weights of the edges are the multiplicities of the edges in the original multigraph. As in~\cite{devosetal} and~\cite{dvorakmohar}, this reformulation, carried out in Section~\ref{equivalent}, turns out to greatly simplify the discussion and the proofs. The equivalent form of Theorem~\ref{maintheorem}, namely Theorem~\ref{maintheorem2}, is then proved in Section~\ref{prooftheorem}, the core of this paper. Finally, we present some concluding remarks and open questions in Section~\ref{concludingremarks}.

\section{Proof of Lemma~\ref{thelemma}}\label{prooflemma}

We use  perfect rubber bands,
a technique inspired by the work of Tutte~\cite{tutte}.

Make every edge a perfect rubber band and pin vertices $u$ and $v$ on
a board. Since $G$ is 2-connected, and with use of Menger's theorem,
every other vertex $w$ admits two vertex-disjoints paths, one linking
$w$ with $u$ and the other linking $w$ with $v$. Therefore, every
vertex will lie on the segment $[u,v]$ on the board. Since
$\mathbb{Q}$ is dense in $\mathbb{R}$, we may modify vertices
positions so that all edge lengths are rational (since they all are
barycentric coordinates, it suffices to have a rational distance
between $u$ and $v$). We may also modify these coordinates locally so
that no two vertices lie at the same spot.

Therefore, every edge lies on a shortest path from $u$ to $v$.

In the end, we may multiply every length by the least common multiple
of the denominators, so that we get an integer length function on the
edges meeting our requirements.\hfill {$\square$}

\section{Reformulating Theorem~\ref{maintheorem} in
  terms \\ of weighted graphs}\label{equivalent}

In the context of Theorem~\ref{maintheorem}, let $G$ be a simple graph
which we seek to make crossing-critical by multiplying (some or
all of) its edges.  With this in mind, let $\overline{G}$ be a
multigraph (that is, a graph with multiedges allowed) whose
underlying simple graph is $G$. Now consider the (positive integer)
weight assignment $w$ on $E(G)$ defined as follows: for each edge $uv$
of $G$, let $w(uv)$ be the number of edges in $\overline{G}$ whose
endpoints are $u$ and $v$ (i.e., the {\em multiplicity} of $uv$).

If we extend the definition of crossing number to weighted graphs,
with the condition that a crossing between two edges contributes to
the total crossing number by the products of their weights, then, from
the crossing numbers point of view, clearly $(G,w)$ captures all the
relevant information from $\overline{G}$.  In particular,
$\cro(\overline{G})= \cro(G,w)$. Moreover, by extending the definition
of crossing-criticality to weighted graphs in the obvious way (which
we now proceed to do), it will follow that $\overline{G}$ is
crossing-critical if and only if $(G,w)$ is crossing-critical.

To this end, let $G$ be a graph and $w$ a positive integral weight assignment on $G$. An edge $e$ of $(G,w)$ is {\em crossing-critical} if $\cro(G,w_e) < \cro(G,w)$, where $w_e$ is the weight assignment defined by $w_e(f) = w(f)$ for $f\neq e$ and $w_e(e) = w(e)-1$. As with ordinary graphs, $(G,w)$ is {\em crossing-critical} if all its edges are crossing-critical.

Under this definition of crossing-criticality for weighted graphs, it is now obvious that if we start with a multigraph $\overline{G}$ and derive its associated weighted graph $(G,w)$ as above, then $\overline{G}$ is crossing-critical if and only if $(G,w)$ is crossing-critical.

In view of this equivalence (for crossing number purposes) between multigraphs and weighted graphs, it follows that Theorem~\ref{maintheorem} is equivalent to the following:

\begin{thm}[{\bf Equivalent to Theorem~\ref{maintheorem}}]\label{maintheorem2}
Let $G$ be a near-planar simple graph, with an edge $uv$ such that $G-uv$ is a cubic polyhedral graph. Suppose that $G -\{u,v\}$ is internally $3$-connected. Then there exists a positive integer weight assignment $w$ such that $(G,w)$ is crossing-critical.
\end{thm}






For the rest of this paper:

\begin{itemize}
\item we let
$
G_{u,v}:= G-\{u,v\}
$; and
\item we refer to the hypotheses that $G-uv$ is
$3$-connected and $G_{u,v}$ is internally $3$-connected 
simply as the {\em connectivity assumptions} on $G-uv$ and
$G_{u,v}$, respectively.
\end{itemize}

\section{Some facts on $G_{u,v}$ and its dual $G^*_{u,v}$}

Before moving on to the proof of Theorem~\ref{prooftheorem}, we
establish some facts on the graph $G_{u,v}$.

\subsection{Remarks on the vertices incident with $u$ and $v$}

Let $u_1, u_2,$ and $u_3$ be the vertices of $G$ (other than
$v$) adjacent to $u$. Analogously, 
Let $v_1, v_2,$ and $v_3$ be the  vertices of $G$ (other than
$u$) adjacent to $v$.

We start by noting that $u_1, u_2, u_3, v_1, v_2, v_3$ are all
distinct. First of all, if $i\neq j$, then since $G$ is simple it
follows that $u_i\neq u_j$. Now suppose that $u_i=v_j$ for some $i,j$,
and consider an embedding of $G-uv$ in the plane. It is easy to see
that since $u_i=v_j$, and $G-uv$ is cubic, it follows that $uv$ can be
added to the embedding of $G-uv$ without introducing any crossings,
resulting in an embedding of $G$. This contradicts the nonplanarity of
$G$.  Thus $u_i\neq v_j$ for all $i,j\in \{1,2,3\}$, completing the
proof that $u_1, u_2, u_3, v_1, v_2, v_3$ are all distinct.

%
\subsection{
The embeddings of $G-uv$ and $G_{u,v}$
}
\bigskip

We note that the connectivity assumptions on $G-uv$ and $\gde$ imply that these two graphs admit unique (up to homeomorphism) embeddings in the plane.  This allows us, for the rest of the proof, to regard these as graphs embedded in the plane.

Since $\gde$ is a subgraph of $G-uv$, it follows that we may assume
that the restriction of the embedding of $G-uv$ to $\gde$ is precisely
the embedding of $\gde$.  Conversely, to obtain the embedding of
$G-uv$, we may start with the embedding of $\gde$; then we find the
(unique) face $F_u$ incident with $u_1, u_2$, and $u_3$, and
draw $u u_1, u u_2$, and $u u_3$ (and, of course, $u$) inside $F_u$;
and similarly
find the
(unique) face $F_v$ incident with $v_1, v_2$, and $v_3$, and
draw $v v_1, v v_2$, and $v v_3$ (and, of course, $v$) inside $F_v$.
Note that $F_u\neq F_v$, as otherwise the edge $uv$
could be added to the embedding of $G-uv$ without introducing any
crossings, resulting in an embedding of $G$, contradicting its
nonplanarity.

\subsection{Weight assignments on the dual $\gd$ of $\gde$}

We shall make extensive use of weight assignments on the dual (embedded
graph) $\gd$ of $\gde$. We start by noting that $\gd$ is well-defined
(and admits a unique plane embedding) since $\gde$ admits a unique
plane embedding.  As with $G-uv$ and $\gde$, this allows us to
unambiguously regard $\gd$ for the rest of the proof as an embedded
graph. We shall let
$\ff$ denote the set of all faces in $\gde$ (equivalently, the set of
all vertices of $\gd$). 

A weight assignment $\wuv$ on $\gde$ naturally induces a weight 
assignment $\wuvd$ on $\gd$, and vice versa:
if $e$ is an edge of $\gde$ and $e^*$ is its dual edge in $\gd$, then
we simply let $\wuvd(e^*)= \wuv(e)$. Trivially, a weight assignment
$\w$ on the whole graph $G$ also naturally induces a weight assignment
$\wuvd$ on $\gd$: it suffices to consider the restriction $\wuv$ of
$\w$ to $\gde$, and from this we obtain $\wuvd$ as we just
described. 

If $\wuvd$ is a weight assignment on $\gd$, then for $F,F'\in\ff$
we let $d_{\wuvd}(F,F')$ denote the length of a shortest $FF'$-path in $\gd$
under  $\wuvd$. We  call $d_{\wuvd}(F,F')$ the {\em distance}
between $F$ and $F'$ under $\wuvd$. 


Now since for $i=1,2,3$ the vertex $u_i$ has degree $2$ in
$\gde$, it follows that $u_i$ is incident with exactly two faces in
$\gd$, one of which is $F_u$; let $F_{u_i}$ denote the other face. Thus it
makes sense to define the {\em distance} $d_{\wuvd}(u_i,F)$ between $u_i$
and any face $F\in\ff$ as
$\min\{d_\wuvd(F_u,F),d_\wuvd(F_{u_i},F)\}$.  We define $F_{v_i}$ and
$d_{\wuvd}(v_i,F)$ analogously, for $i=1,2,3$.

The connectivity assumption on $\gde$ ensures that  $F_{u_1}, F_{u_2}$ and
$F_{u_3}$ are pairwise distinct. Similarly,  $F_{v_1}, F_{v_2}$ and
$F_{v_3}$ are pairwise distinct. 
Note that maybe $F_{u_i} = F_v$ for some $i\in\{1,
2, 3\}$, or $F_{v_j} = F_u$ for some $j\in\{1, 2,
3\}$.

Finally, we say that a weight assignment $\wuvd$ on $\gd$ is {\em balanced} if each edge $e^*$ of $\gd$ belongs to a shortest $F_uF_v$-path in $(\gd,\wuvd)$.




\section{Proof of Theorem~\ref{maintheorem2}}\label{prooftheorem}

First we show (Proposition~\ref{pro:ifthere}) that if there exists a
weight assignment $\omega$ on $G$ with certain properties, then
$(G,\omega)$ is crossing-critical. The existence of a weight
assignment with these properties is established in 
Proposition~\ref{pro:thereis}, and so Theorem~\ref{maintheorem2}
immediately follows.

\begin{proposition}\label{pro:ifthere}
Suppose that $\omega$ is a positive integer weight assignment on $G$ with the following
properties:
\begin{description}
\item{(1)} The induced weight assignment $\omega^*$ on $\gd$ is balanced.
\item{(2)} For every pair of edges $e,e'$ of $\gde$, 
$\omega(e)  \omega(e') > d_{\omega^*}(F_u,F_v)  \cdot \omega(uv)$.
\item{(3)} $d_{\omega^*}(u_1,F)\cdot\omega(uu_1) +
  d_{\omega^*}(u_2,F)\cdot\omega(uu_2) +
  d_{\omega^*}(u_3,F)\cdot\omega(uu_3) \ge  d_{\omega^*}(F_u,F)\cdot \omega(uv)$, for
every $F\in\ff$.
\item{(4)} For each $i=1, 2, 3$, there is a face
  $U_i\in\ff$ such that $d_{\omega^*}(u_i,U_i)>0$ and
  $d_{\omega^*}(u_1,U_i)\cdot \omega(uu_1)  +  d_{\omega^*}(u_2,U_i)
  \cdot \omega(uu_2) +
d_{\omega^*}(u_3, U_i) \cdot \omega(uu_3) = d_{\omega^*}(F_u,U_i)
\cdot \omega(uv)$.
\item{(5)} $d_{\omega^*}(v_1,F)\cdot\omega(vv_1) +
  d_{\omega^*}(v_2,F)\cdot\omega(vv_2) +
  d_{\omega^*}(v_3,F)\cdot\omega(vv_3) \ge d_{\omega^*}(F_v,F)\cdot \omega(uv)$, for
every $F\in\ff$.
\item{(6)} For each $i=1, 2, 3$, there is a face
  $V_i\in\ff$ such that $d_{\omega^*}(v_i,V_i)>0$ and
  $ d_{\omega^*}(v_1,V_i)\cdot \omega(vv_1)  + d_{\omega^*}(v_2,V_i)
  \cdot \omega(vv_2) +
d_{\omega^*}(v_3, V_i)\cdot \omega(vv_3) = d_{\omega^*}(F_v,V_i)\cdot \omega(uv)$.
\item{(7)} For all $i,j\in\{1,2,3\}$,
$
\omega(uu_i) \cdot \omega(vv_j) < (1/9) \min\{\ \omega(e)      \  | \ e \in E(\gde)\}. 
$
\end{description}
Then $(G,\omega)$ is crossing-critical.
\end{proposition}


\begin{proof}

Throughout the proof, for brevity we let $t:=\omega(uv)$. 

To help comprehension, we break the proof into several steps.

\bigskip
\noindent{(A) } {\em $\cro(G,\omega) \le t\cdot d_{\omega^*}(F_u,F_v)$. 
}
\bigskip

Start with the (unique) embedding of $G-uv$, and draw $uv$ following a
shortest $F_uF_v$-path in  $(\gd,\omega^*)$.
Then the sum of the weights of the edges
crossed by $uv$ equals the total weight of the shortest $F_uF_v$-path, that is, $d_{\omega^*}(F_u,F_v)$ (here we
use the elementary, easy to check fact that crossings between adjacent
edges can always be avoided; in this case, we may draw $uv$ so that it
crosses no edge adjacent to $u$ or $v$). Since $\omega(uv)=t$, it follows
that such a drawing of $(G,\omega)$ has exactly 
$t\cdot d_{\omega^*}(F_u,F_v)$.
crossings. 

\bigskip
\noindent{(B) } {\em 
$\cro(G,\omega) = t\cdot d_{\omega^*}(F_u,F_v)$.
}
\bigskip

Consider a crossing-minimal drawing $\dd$ of $(G,\omega)$. An immediate consequence of (2) and
(A) is that the drawing of $\gde$ induced by $\dd$ is an
embedding (that is, no two edges of $\gde$ cross each other in
$\dd$). 

Now let $F'$ (respectively, $F''$) denote the face of $\gde$ in which
$u$ (respectively, $v$) is drawn in $\dd$. Clearly, for $i=1,2,3$ the
edge $uu_i$ contributes in at least 
$\omega(uu_i)\cdot d_{\omega^*}(u_i,F')$ crossings. Analogously, 
for $i=1,2,3$ the
edge $vv_i$ contributes in at least 
$\omega(vv_i)\cdot d_{\omega^*}(v_i,F'')$ crossings. Thus it follows
  from (3) and (5) that the edges in $\{uu_1, uu_2, uu_3, vv_1,
  vv_2, vv_3\}$ contribute in at least 
$t\cdot d_{\omega^*}(F_u,F') + t\cdot d_{\omega^*}(F_v,F'') = 
t\cdot (d_{\omega^*}(F_u,F') + d_{\omega^*}(F_v,F''))$ crossings. On
the other hand, since the ends $u,v$ of $uv$ are in faces $F'$ and
$F''$, it follows that edge $uv$ contributes in at least $t\cdot d_{\omega^*}(F',F'')$ crossings. We conclude that
$\dd$ has at least $t\cdot 
\bigl( d_{\omega^*}(F_u,F') + d_{\omega^*}(F_v,F'') + d_{\omega^*}(F',F'')\bigr)$.
Elementary triangle inequality arguments show that
$d_{\omega^*}(F_u,F') + d_{\omega^*}(F_v,F'') + d_{\omega^*}(F',F'')
\ge d_{\omega^*}(F_u,F_v)$, and so
$\dd$ has at least  $t\cdot d_{\omega^*}(F_u,F_v)$ crossings. 
Thus $\cro(G,\omega) \ge t\cdot d_{\omega^*}(F_u,F_v)$.  The reverse
inequality is given in (A), and so (B) follows.

\bigskip
\noindent{(C) } {\em Crossing-criticality of the edges in $\gde$ and
  of the edge $uv$.}
\bigskip

Let $e$ be any edge in $\gde$. 
We proceed similarly as in (A). Start with the (unique) embedding of
$G-uv$, and draw $uv$ following a shortest $F_uF_v$-path in
$(\gd,\omega^*)$ that includes $e^*$ (the existence of such a path is
guaranteed by the balancedness of $\omega^*$). This yields a drawing of
$(G,\omega)$ with exactly
$t\cdot d_{\omega^*}(F_u,F_v)$ crossings, in which $e$ and $uv$ cross each
other. Since $\cro(G,\omega) = 
t\cdot d_{\omega^*}(F_u,F_v)$, it follows that $e$ and $uv$ are both
crossed in a crossing-minimal drawing of $(G,\omega)$. Therefore both
$e$ and $uv$ are crossing-critical in $(G,\omega)$. 

\bigskip
\noindent{(D) } {\em Crossing-criticality of the edges $uu_1, uu_2,
  uu_3, vv_1, vv_2$, and $vv_3$.}
\bigskip

We prove the criticality of $uu_1$; the proof of the criticality of the other edges is
totally analogous. 

Consider the (unique) embedding of $\gde$. Put $u$ in face $U_i$
(see property (4)) and $v$ in face $F_v$. Then  draw $uu_j$, for
$j=1,2,3$, adding $\omega(uu_j)\cdot d_{\omega^*} (u_j,U_1)$
crossings with the edges in $\gde$. Since crossings between adjacent
edges can always be avoided, it follows that $uu_1, uu_2, uu_3$ get drawn
by adding $
\omega(uu_1)\cdot d_{\omega^*} (u_1,U_1) 
+
\omega(uu_2)\cdot d_{\omega^*} (u_2,U_1) 
+
\omega(uu_3)\cdot d_{\omega^*} (u_3,U_1) 
= t\cdot d_{\omega^*}(F_u,U_1)
$ 
crossings (using (4)).
Finally we draw $vv_1, vv_2, vv_3$ in face $F_v$. Now this last step
may add crossings, but only of the edges $vv_1, vv_2, vv_3$ with the
edges $uu_1, uu_2, uu_3$.  In view of (7), the last step added fewer than
$9\cdot(1/9) \min\{\ \omega(e)      \  | \ e \in E(\gde)\}
= \min\{\ \omega(e)      \  | \ e \in E(\gde)\}$ crossings.
 We finally draw $uv$; since $u$ is in face $U_1$ and
$v$ is in face $F_v$, it follows that $uv$ can be drawn by adding
$t \cdot d_{\omega^*}(U_1,F_v)$ 
crossings. 

The described drawing $\dd$ of $G$ has then fewer than $t\cdot
d_{\omega^*}(F_u,U_1) + t \cdot d_{\omega^*}(U_1,F_v) +
\min\{\ \omega(e) \ | \ e \in E(\gde)\} = t\cdot d_{\omega^*}(F_u,F_v)
+ \min\{\ \omega(e) \ | \ e \in E(\gde)\} = \cro(G,\omega) +
\min\{\ \omega(e) \ | \ e \in E(\gde)\}$ crossings, where for the
first equality we used the balancedness of $\omega^*$, and for the
second equality we used (B).  Thus $\cro(\dd) < \cro(G,\omega) +
\min\{\ \omega(e) \ | \ e \in E(\gde)\}$.

In $\dd$, the edge $uu_1$ contributes in 
$\omega(uu_1)\cdot d_{\omega^*} (u_1,U_1)$ crossings; note that (4)
implies that $\omega(uu_1)\cdot d_{\omega^*} (u_1,U_1)>0$.  
Since obviously $d_{\omega^*} (u_1,U_i)\ge \min\{\omega(e)      \  |
\ e \in E(\gde)\}$, 
it follows that $uu_1$ contributes in 
at least $\min\{\omega(e)      \  | \ e \in E(\gde)\}$
crossings. Thus, 
if we remove $uu_1$ we obtain a drawing of $G-uu_1$ with 
fewer than $\cro(G,\omega)$ crossings. Therefore
$uu_1$ is critical
in $(G,\omega)$, as claimed. 
\end{proof}


\begin{proposition}\label{pro:thereis}
There exists a positive integer weight assignment $\omega$ on $G$ that satisfies
(1)--(7) in Proposition~\ref{pro:ifthere}
\end{proposition}

\begin{proof}
We start with a balanced positive integer weight assignment $\mu^*$ on
$\gd$. The  
existence of such a $\mu^*$
is guaranteed from Lemma~\ref{thelemma}, which
applies since the connectivity assumption on $\gde$ implies that $\gd$
is also $3$-connected. Let $\mu$ denote the (positive integer) weight
assignment naturally induced on $\gde$. 


\vglue 0.5 cm
\noindent{\bf Claim I. } {\em There exists a rational point
  $(x_1,x_2,x_3)$ such that, for every $F\in\ff\setminus\{F_u\}$,
$$d_{\mu^*}(u_1,F)x_1 + d_{\mu^*}(u_2,F)x_2 + d_{\mu^*}(u_3,F)x_3 \ge
d_{\mu^*}(F_u,F).$$
Moreover, there exist faces $U_1, U_2, U_3$ of $\gde$, such that
  $d_{\mu^*}(u_i,U_i) >0$ and, for $i=1,2,3$, 
$$d_{\mu^*}(u_1,U_i)x_1 +
d_{\mu^*}(u_2,U_i)x_2 +
d_{\mu^*}(u_3,U_i)x_3 = d_{\mu^*}(F_u,U_i).$$
} 
\vglue 0.5 cm

\begin{proof}
Let $\ii$ be the system of inequalities 
$\ii:=\{d_{\mu^*}(u_1,F)x_1 + d_{\mu^*}(u_2,F)x_2 + d_{\mu^*}(u_3,F)x_3 \ge
d_{\mu^*}(F_u,F)\ | \ f\in \ff\setminus\{F_u\}\}$, and let 
$\Lambda$ denote the set of all triples $(x_1,x_2,x_3)$ that satisfy
all inequalities in $\ii$.
Clearly, if $x_1, x_2$, and $x_3$ are all large enough
then $(x_1, x_2, x_3)$ is in $\Lambda$. Thus $\Lambda$ is a nonempty
convex polyhedron in $\real^3$.

Each inequality $J$ in $\ii$ naturally defines a plane in $\real^3$,
namely the plane obtained by substituting $\ge$ with $=$. We call this
the {\em plane associated to} $J$.

Since each $u_i$ is incident with face $F_{u_i}$, it follows that 
$d_{\mu^*}(u_1,F_{u_1}) = d_{\mu^*}(u_2,F_{u_2}) =
d_{\mu^*}(u_3,F_{u_3}) = 0$,  Therefore the inequalities corresponding to $F_{u_1},
F_{u_2}$ and $F_{u_3}$, respectively, define the following system
$\Gamma$:

\begin{tabular}{cccccccr}\label{tab:tab}
 & & $d_{\mu^*}(u_2,F_{u_1})x_2$ & $+$ & $d_{\mu^*}(u_3,F_{u_1})x_3$ &
  $\geq$ & $d_{\mu^*}(F_u,F_{u_1})$ & ($\Gamma 1$)\\ \label{linsysu1} \nonumber 
$d_{\mu^*}(u_1,F_{u_2})x_1$ &  &  & $+$ & $d_{\mu^*}(u_3,F_{u_2})x_3$ & $\geq$ &
  $d_{\mu^*}(F_u,F_{u_2})$ & ($\Gamma 2$)\\ 
$d_{\mu^*}(u_1,F_{u_3})x_1$ & $+$ & $d_{\mu^*}(u_2, F_{u_3})x_2$ &  &  & $\geq$ &
  $d_{\mu^*}(F_u,F_{u_3})$ & ($\Gamma 3$)\\  \label{linsysu3} \nonumber 
\end{tabular}

\def\gone{{($\Gamma 1)$}}
\def\gtwo{{($\Gamma 2)$}}
\def\gthree{{($\Gamma 3)$}}

\noindent where all coefficients  (since the faces $F_{u_1}, F_{u_2},
F_{u_3}$ are pairwise distinct) and all the right-hand sides are
strictly positive integers. We refer to this as the {\em positive integrality} property
of \gone, \gtwo, and \gthree.

This positive integrality property implies that the planes associated to
\gone\ and
\gtwo\ intersect in a line whose intersection with the positive octant
is a (full one-dimensional) segment. In
particular, there is a point $(a_1,a_2,a_3)$, all of whose points are
positive and rational, and that lies on the intersection of the planes
associated to \gone\ and \gtwo. 

\def\or{{r^{\rightarrow}}}

Now consider the set of points $\or:=\{(x_1, a_2, a_3)\ | \ x_1 \ge 0\}$. Thus $\or$ is a ray starting at $(0, a_2, a_3)$ and
parallel to the $x_1$-axis, lying on the plane associated to \gone. Now
since \gone\ is the only inequality in $\ii$ whose $x_1$ coefficient is
zero, it follows that for every large enough $x_1$, the point $(x_1,
a_2, a_3)$ is in $\Lambda$. Now every inequality in $\ii$ distinct
from \gone\ intersects $\or$ in at most one point.  Let $a_1'$ be
largest possible such that $(a_1',a_2,a_3)$ is the intersection of
$\or$ with a plane associated to an inequality $I$ in
$\ii\setminus\{$\gone$\}$; since $a_1$ is the intersection of \gtwo\
with \gone, it follows that $a_1'$ is well-defined. 

Since $a_1',a_2,a_3$ all arise from the intersection of planes with
integer coefficients, it follows that they are all rational
numbers. We claim that the rational point $(a_1',a_2,a_3)$ satisfies
the requirements in the Claim.

Consider any inequality $J$ in $\ii$ distinct from \gone, \gtwo, and
\gthree. Thus all the coefficients of $J$ are nonzero, and so the
intersection of the plane associated to $J$ with the positive octant is a
triangle. Let $(j,a_2,a_3)$ be the intersection of this triangle with
the line
$\{(x_1,a_2,a_3)\ | \ x_1\in \real\}$. Then $j \le a_1'$, and so the ray
$\{(x_1,a_2,a_3)\ | \ x_1\in \real, x_1 \ge j\}$ is contained in the
{\em feasible} region of $J$ (that is, the region of $\real^3$ that
consists of those points for which $J$ is satisfied). In particular, $(a_1',a_2,a_3)$ is in the
feasible region of $J$.  A similar 
argument shows that also in the case in which $J$ is either
\gtwo\ or \gthree, then $(a_1',a_2,a_3)$ is in the feasible region of
$J$. This proves the first part of Claim I.

Now we recall that $(a_1',a_2,a_3)$ lies on the plane associated to
\gone, that is,
 $d_{\mu^*}(u_2,F_{u_1})a_2+d_{\mu^*}(u_3,F_{u_1})a_3 =
d_{\mu^*}(F_u,F_{u_1})$. Since $d_{\mu^*}(u_1,F_{u_1})=0$, we have
 $d_{\mu^*}(u_1,F_{u_1})a_1'+d_{\mu^*}(u_2,F_{u_1})a_2+d_{\mu^*}(u_3,F_{u_1})a_3 =
d_{\mu^*}(F_u,F_{u_1})$.  Noting that $d_{\mu^*}(u_2,F_{u_1})>0$ and
$d_{\mu^*}(u_e,F_{u_1})>0$, the second part of Claim I follows for
$i=2$ and $3$ by setting $U_2=U_3=F_{u_1}$. Now let $U_1$ be the face
in $\gg$ associated to inequality $I$. Thus
$d_{\mu^*}(u_1,U_1)a_1'+d_{\mu^*}(u_2,U_1)a_2+d_{\mu^*}(u_3,U_1)a_3 =
d_{\mu^*}(F_u,U_1)$. We recall that \gone\ is the only inequality in
$\ii$ whose $x_1$ coefficient is $0$; since inequality $I$ is distinct from \gone\, it
follows that $d_{\mu^*}(u_1,U_1)>0$. Thus the second part of Claim I
follows for $i=1$ for this choice of $U_1$.
\end{proof}

The proof of the following statement is totally analogous:

\vglue 0.5 cm
\noindent{\bf Claim II. } {\em There exists a rational point
  $(y_1,y_2,y_3)$ such that, for every $F\in\ff\setminus\{F_v\}$,
$$d_{\mu^*}(v_1,F)y_1 + d_{\mu^*}(v_2,F)y_2 + d_{\mu^*}(v_3,F)y_3 \ge
d_{\mu^*}(F_v,F).$$
Moreover, there exist faces $V_1, V_2, V_3$ of $\gde$, such that
  $d_{\mu^*}(v_i,V_i) >0$ and, for $i=1,2,3$, 
$$d_{\mu^*}(v_1,V_i)y_1 +
d_{\mu^*}(v_2,V_i)y_2 +
d_{\mu^*}(v_3,V_i)y_3 = d_{\mu^*}(F_u,V_i).$$
}
\vglue -1.4 cm \hfill$\square$

\vglue 0.8 cm

Let $(p_1/q_1, p_2/q_2, p_3/q_3)$ be a point as in Claim I, and let
$(a_1/b_1, a_2/b_2, a_3/b_3)$ be a point as in Claim II, 
where
all $p_i$s, $q_i$s, $a_i$s, and $b_i$s are integers. Let $M:=q_1q_2q_3b_1b_2b_3$, and let
$r_1:=p_1q_2q_3b_1b_2b_3,
r_2:=p_2q_1q_3b_1b_2b_3$, 
$r_3:=p_3 q_1 q_2b_1b_2b_3$,
$s_1:=a_1b_2b_3q_1q_2q_3,
s_2:=a_2b_1b_3q_1q_2q_3$, and 
$s_3:=a_3 b_1 b_2q_1q_2q_3$.

Then $(r_1, r_2, r_3)$ is a positive integer
solution to the set of inequalities 
$\{d_{\mu^*}(u_1,F)r_1 + d_{\mu^*}(u_2,F)r_2 + d_{\mu^*}(u_3,F)r_3 \ge
M\cdot d_{\mu^*}(F_u,F) : F\in\ff\setminus\{F_u\}\}$, and for each
$i=1,2,3$, we have
$d_{\mu^*}(u_1,U_i)r_1 +
d_{\mu^*}(u_2,U_i)r_2 +
d_{\mu^*}(u_3,U_i)r_3 = M\cdot d_{\mu^*}(F_u,U_i)$. 

Similarly, $(s_1, s_2, s_3)$ is a positive integer
solution to the set of inequalities 
$\{d_{\mu^*}(v_1,F)s_1 + d_{\mu^*}(v_2,F)s_2 + d_{\mu^*}(v_3,F)s_3 \ge
M\cdot d_{\mu^*}(F_v,F) : F\in\ff\setminus\{F_v\}\}$, and for each
$i=1,2,3$, we have
$d_{\mu^*}(v_1,V_i)s_1 +
d_{\mu^*}(v_2,V_i)s_2 +
d_{\mu^*}(v_3,V_i)s_3 = M\cdot d_{\mu^*}(F_v,V_i)$.

Finally, let $c$ be any integer greater than $ M\cdot
d_{\mu^*}(F_u,F_v)/(\min\{\mu(e)\ | \ e\in E(\gde)\})^2 $ and also greater
than $ 9r_i s_j/\min\{\mu(e)\ | \ e\in E(\gde)\}
$,
for all $i,j \in\{1,2,3\}$. 


Define the weight assignment $\omega$ on $G$ as follows: 

\begin{itemize}

\item $\omega(uv) = M$;

\item $\omega(uu_i)=r_i$ and $\omega(vv_i)=s_i$ for $i=1,2,3$;

\item $\omega(e) = c\cdot \mu(e)$, for all edges $e$ in $\gde$. 

\end{itemize}

We claim that $\omega$ (and its induced weight assignment $\omega^*$
on $\gd$) satisfies (1)--(7) in
Proposition~\ref{pro:ifthere}.

To see that $\omega^*$ satisfies (1), it suffices to note that $\omega^*$
inherits the balancedness (when restricted to $\gd$) from $\mu^*$. 

Now let $e,e'$ be edges of $\gde$. Then $
\omega(e)
\omega(e') 
=
c^2\cdot \mu(e)\mu(e') \ge c^2\cdot  (\min\{ 
\mu(f)\ | \ f\in E(\gde)\
\})^2 > c\cdot M \cdot d_{\mu^*}(F_u,F_v) = \omega(uv)(c\cdot
d_{\mu^*}(F_u,F_v)) = \omega(uv)\cdot d_{\omega^*} (F_u,F_v)$.  This
proves (2). 


For (3) and (4), recall that $(r_1, r_2, r_3)$ is a positive integer
solution to the set of inequalities 
$\{d_{\mu^*}(u_1,F)r_1 + d_{\mu^*}(u_2,F)r_2 + d_{\mu^*}(u_3,F)r_3 \ge
M\cdot d_{\mu^*}(F_u,F) : F\in\ff\setminus\{F_u\}\}$, and for each
$i=1,2,3$, we have
$d_{\mu^*}(u_1,U_i)r_1 +
d_{\mu^*}(u_2,U_i)r_2 +
d_{\mu^*}(u_3,U_i)r_3 = M\cdot d_{\mu^*}(F_u,U_i)$. The definition of
$\omega$ (and its induced $\omega^*$) then immediately imply (3) and
(4) (we are using that for any faces $F,F'\in\ff\setminus\{F_u\}$, 
$d_{\omega^*}(F,F') = c\cdot d_{\mu^*}(F,F')$).
The proof that (5) and (6) hold is totally analogous.

Finally, we recall that we defined $c$ so that $ c > 9r_i
s_j/\min\{\mu(e)\ | \ e\in E(\gde) \} $ for all $i,j
\in\{1,2,3\}$. By the definition of $\omega$, this is equivalent to
$ c\cdot \min\{\mu(e) \ | \ e\in E(\gde) \} >
9\omega(uu_i)\omega(vv_j)
$, that is,  
$ \min\{\omega(e) \ | \ e\in E(\gde) \} >
9\omega(uu_i)\omega(vv_j)
$, which is in turn obviously equivalent to (7).
\end{proof}


\section{Concluding Remarks and Open Questions}\label{concludingremarks}

Let $\gg$ be the class of graphs that can be made crossing-critical by
a suitable multiplication of edges.  In this work we have proved that
a large family of graphs is contained in $\gg$ (note that the cubic
condition is only used around vertices $u,v,u_1,u_2, u_3, v_1,v_2$ and
$v_3$; other vertices can have arbitrary degrees).  Which other graphs
belong to $\gg$? Is there any hope of fully characterizing $\gg$?

It is not difficult to prove that we can restrict our attention to
simple graphs: if $\overline{G}$ is a graph with multiple
edges and $G$ is a maximal simple graph contained in $G$, then  
$\oG$ is in $\gg$ if and only if $G$ is in $\gg$.

Jes\'us Lea\~nos has observed that the graph $K_{3,3}^+$ obtained by
adding to $K_{3,3}$ an edge
(between vertices in the same chromatic class) is not in 
$\gg$. 
Following \v{S}ir\'a\v{n}~\cite{siran1,siran2}, an edge $e$ in
a graph $G$ is a {\em Kuratowski edge} if there is a subgraph $H$ of
$G$ that contains $e$ and is homeomorphic to a Kuratowski graph (that is, $K_{3,3}$ or $K_5$). It is trivial to see that the added edge in Lea\~nos's
example is not a Kuratowski edge of $K_{3,3}^+$.  This observation naturally gives
rise to the following.

\begin{conj}\label{con:thecon}
If $G$ is a graph all whose edges are Kuratowski edges, then $G$ can
be made crossing-critical by a suitable multiplication of its edges. 
\end{conj}

We remark that the converse of this statement is not true: 
\v{S}ir\'a\v{n}~\cite{siran1} gave examples of graphs that contain crossing-critical edges
that are not Kuratowski edges. 

The only positive result we have in this direction is that Kuratowski edges can
be made individually crossing-critical: 

\begin{proposition}
If $e$ is a Kuratowski edge of a graph $G$, then $e$ can be made
crossing-critical by a suitable multiplication of the edges of $G$.
\end{proposition}

\begin{proof}
Let $H$ be a subgraph of $G$,
homeomorphic to a Kuratowski graph, such that $e$ is
in $H$. Let $f$ be another edge of $H$ such that there is a drawing
$\dd_H$ of
$H$ with exactly one crossing, which involves $e$ and $f$. Extend
$\dd_H$ to a drawing $\dd$ of $G$. Let $p$ be the number of crossings
in $\dd$.  If $p=1$ then $e$ is already critical in $G$, so there is
nothing to prove. Thus we may assume that $p \ge 2$. 
Add $p^2-1$ parallel edges to each of $e$ and $f$, add $p^4-1$ parallel
edges to all edges in $H\setminus \{e,f\}$, and do not add any
parallel edge to the other edges of $G$. Let $G'$ denote the
resulting graph. 

We claim that $\cro(G')\le p^5$. To see this, consider the drawing
$\dd'$ of $G'$ naturally induced by $\dd$. It is easy to check that each crossing from $\dd$ yields at most
$p^4$ crossings in $\dd'$ (here we use that $e$ and $f$ are the only
edges in $H$ that cross each other in $\dd$). Thus $\dd'$ has at most
$p\cdot p^4=p^5$ crossings, and so 
 $\cro(G') \le p^5$, as claimed. 

On the other hand, it is clear that a drawing of $G'$ in which $e$ and $f$ do not
cross each other has at least $p^6$ crossings. Since $p^6 > p^5 \ge
\cro(G')$, it follows that no such drawing can be optimal. Therefore
$e$ and $f$ cross each other in every optimal drawing of $G'$. This
immediately implies that $e$ is critical in $G'$. 
\end{proof}

The immediate next step towards Conjecture~\ref{con:thecon} seems
already difficult enough so as to prompt us to state it:

\begin{conj}
Suppose that $e,f$ are Kuratowski edges of a graph $G$. Then there
exists a graph $H$, obtained by multiplying edges of $G$, such that
both $e$ and $f$ are crossing-critical in $H$. 
\end{conj}

\vglue 0.8 cm
\noindent{\bf\Large Acknowledgements} 
\vglue 0.8 cm

We thank Jes\'us Lea\~nos for helpful discussions.






\end{document}